\documentclass{amsart}

\usepackage[table,xcdraw]{xcolor}
\usepackage{amssymb}
\usepackage{amsmath}
\usepackage{amsfonts}
\usepackage{amsxtra}
\usepackage{amsthm}
\usepackage{tikz}
\usepackage{tikz-cd}
\usepackage{float}
\usepackage{booktabs}
\usepackage{color}
\usepackage{xcolor}
\usepackage{graphicx}
\usepackage{enumitem}
\usepackage{venndiagram}
\usetikzlibrary{positioning,shapes,fit,arrows}
\usepackage{multirow}
\usepackage{multicol}
\usepackage{url}
\usepackage{hyperref}
\usepackage{listings}

\newtheorem{theorem}{Theorem}[section]
\newtheorem{proposition}[theorem]{Proposition}
\newtheorem{corollary}[theorem]{Corollary}
\newtheorem{lemma}[theorem]{Lemma}
\newtheorem{exmp}[theorem]{Example}
\newtheorem{definition}[theorem]{Definition}

\newtheorem*{question}{Question}

\newcommand{\abf}{\mathbf{a}}

\newcommand{\M}{\mathrm{M}}

\newcommand{\Id}{\mathrm{Id}}

\def\n{\mathbb{N}}
\def\z{\mathbb{Z}}
\def\q{\mathbb{Q}}
\def\r{\mathbb{R}}
\def\cc{\mathbb{C}}

\def\inv{^{-1}}

\definecolor{codegreen}{rgb}{0,0.6,0}
\definecolor{codegray}{rgb}{0.5,0.5,0.5}
\definecolor{codepurple}{rgb}{0.58,0,0.82}
\definecolor{backcolour}{rgb}{0.95,0.95,0.92}

\lstdefinestyle{mystyle}{
    backgroundcolor=\color{backcolour},   
    commentstyle=\color{codegreen},
    keywordstyle=\color{magenta},
    numberstyle=\tiny\color{codegray},
    stringstyle=\color{codepurple},
    basicstyle=\ttfamily\footnotesize,
    breakatwhitespace=false,         
    breaklines=true,                 
    captionpos=b,                    
    keepspaces=true,                 
    numbers=left,                    
    numbersep=5pt,                  
    showspaces=false,                
    showstringspaces=false,
    showtabs=false,                  
    tabsize=2
}

\begin{document}

\title[]{Artinian Meadows}

\author[]{João Dias}
\author[]{Bruno Dinis}

\address[João Dias]{Departamento de Matemática, Universidade de Évora}
\email{joao.miguel.dias@uevora.pt}

\address[Bruno Dinis]{Departamento de Matemática, Universidade de Évora}
\email{bruno.dinis@uevora.pt}



\begin{abstract}
We introduce the notion of Artinian meadow as an algebraic structure constructed from an Artinian ring which is also a common meadow, i.e.\ a commutative and associative structure with two operations (addition and multiplication) with additive and multiplicative identities and for which inverses are total. The inverse of zero in a common meadow is an error term $\abf$ which is absorbent for addition. We show that, in analogy with what happens with commutative unital Artinian rings, Artinian meadows  decompose as a product of local meadows in an essentially unique way. We also provide a canonical way to construct meadows from unital commutative rings.
\end{abstract}

\maketitle

\section{Introduction}

Bergstra and Tucker introduced in \cite{10.1145/1219092.1219095} the notion of \emph{meadow} as
an algebraic structure with two operations (addition and multiplication) where both addition and multiplication are operations for which the inverses are total. This means in particular that in meadows one is allowed to divide by zero. At the time, the main idea was to view these structures as abstract data types given by equational axiomatizations \cite{bergstra2020arithmetical,Bergstra2008,BP(20),10.1145/1219092.1219095} so that one is able to obtain simple term rewriting systems which are easier to automate in formal reasoning  \cite{bergstra2020arithmetical,bergstra2023axioms}. 
(see also \cite{Bergstra_Tucker_2024, bergstra2024ringscommondivisioncommon} for other recent developments).

An alternative line of research was made possible due to a recently discovered connection between meadows and rings \cite{Dias_Dinis(23)} (see Theorem~\ref{T:DirectedLattice} below). In fact, for the purposes of this paper, and from a purely algebraic point of view, two different classes of meadows are relevant: \emph{common meadows}, introduced by Bergstra and Ponse in \cite{Bergstra2015} and \emph{pre-meadows with $\abf$}, introduced by the authors in \cite{Dias_Dinis(23)}.  In both pre-meadows with $\abf$ and common meadows the inverse of zero is an element (denoted $\abf$) that is absorbent for both operations. This is in contrast, for example, with involutive meadows where the inverse of zero is zero itself \cite{10.1145/1219092.1219095,Dinis_Bottazzi}. The interesting thing about pre-meadows with $\abf$ (and common meadows) is that they can be decomposed as disjoint unions of rings \cite{Dias_Dinis(23)}. When compared with common meadows, pre-meadows with $\abf$ take a step back by not requiring the existence of an inverse function. This turns out to be of relevance in trying to enumerate finite common meadows \cite{Dias_Dinis(24)}. In the finite case, the two classes are even more deeply related since if $P$, a pre-meadow with $\abf$ such that $0\cdot P$ is finite, and the partial order in $0\cdot P$ is a total order, then $P$ is actually a common meadow. Note that since zero is invertible, some usual algebraic properties are not satisfied. For example $0 \cdot x$ may not be equal to $0$.

After recalling all the necessary definitions and results about pre-meadows with $\abf$ and common meadows, in Section~\ref{S:Preliminaries}, 
we present in Section~\ref{S:MeadowstoRings} a canonical construction $M(R)$, of a pre-meadow with $\abf$, starting from a unital commutative ring $R$. We show that in case $R$ is Artinian, the pre-meadow with $\abf$ is actually a common meadow.

In Section~\ref{S:Decomposition} we introduce  a class of common meadows, called \emph{Artinian meadows}, in analogy with the similar notion in rings. We show a decomposition result for Artinian meadows in terms of \emph{local meadows} -- common meadows whose associated directed lattice have a single penultimate vertex. This decomposition is unique up to permutation.

\section{Preliminaries}\label{S:Preliminaries}
    In this section we recall some definitions and results on common meadows.

\begin{definition}
        A \emph{pre-meadow} is a structure $(P,+,-,\cdot)$ satisfying the following equations
    \begin{multicols}{2}
\begin{enumerate}
\item[$(P_1)$] $(x+y)+z=x+(y+z) $
\item[$(P_2)$] $x+y=y+x $ 
\item[$(P_3)$]  $x+0=x$ 
\item[$(P_4)$] $x+ (-x)=0 \cdot x$
\item[$(P_5)$] $(x \cdot y) \cdot z=x \cdot (y \cdot z)$ 
\item[$(P_6)$]  $x \cdot y=y \cdot x $
\item[$(P_7)$] $1 \cdot x=x$
\item[$(P_8)$] $x \cdot (y+z)= x \cdot y + x \cdot z$
\item[$(P_9)$] $-(-x)=x$
\item[$(P_{10})$] $0\cdot(x+y)=0\cdot x \cdot y$
\end{enumerate}
\end{multicols}

 Let $P$ be a pre-meadow and let $P_z:=\{x\in M\mid 0\cdot x=z\}$. We say that $P$ is a \emph{pre-meadow with $\abf$} if there exists a unique $z\in 0\cdot P$ such that $|P_z|=1$ (denoted by $\abf$) and $x+\abf=\abf$, for all $x\in P$. 
    A \emph{common meadow} is a pre-meadow with $\abf$ equipped with an inverse function $(\cdot)\inv$ satisfying
    \begin{multicols}{2}
\begin{enumerate}
\item[$(M_1)$] $x \cdot x^{-1}=1 + 0 \cdot x^{-1}$
\item[$(M_2)$]$(x \cdot y)^{-1} = x^{-1} \cdot y^{-1}$
\item[$(M_3)$] $(1 + 0 \cdot x)^{-1} = 1 + 0 \cdot x $
\item[$(M_4)$] $ 0^{-1}=\abf$
\end{enumerate}
\end{multicols}
\end{definition}
     
\begin{definition}\label{D:Morphism}
    Let $f:M\rightarrow N$ be a function. We say that $f$ is an \emph{homomorphism of (common) meadows} if $M,N$ are common meadows and for all $x,y\in M$
    \begin{enumerate}
        \item $f(x+y)=f(x)+f(y)$.
        \item $f(x\cdot y)=f(x) \cdot f(y)$.
        \item $f(1_M)=1_{N}$.        
    \end{enumerate}
\end{definition}

    \begin{definition}\label{D:LatticeRings}
        A \emph{directed lattice} of rings $\Gamma$ over a countable lattice $L$ consists on a family of commutative rings $\Gamma_i$ indexed by $i\in L$, such that $\Gamma_i$ is a unital commutative ring for all $i\in L\setminus \min(L)$ and $\Gamma_{\min(L)}$ is the zero ring, together with a family of ring homomorphisms $f_{i,j}:\Gamma_j\rightarrow\Gamma_i$ whenever $i<j$ such that $f_{i,j}\circ f_{j,k}=f_{i,k} $ for all $i<j<k$. 
    \end{definition}

Each pre-meadow has an associated directed lattice whose vertices are rings. The ring homomorphisms of the associated lattice are called \emph{transition maps}.

    \begin{exmp}
    Consider the following lattice $L$ over the set $\{1,2,3,4,5\}$
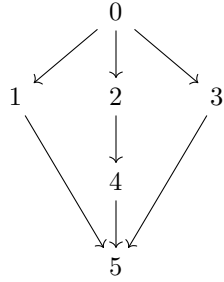
    \[\begin{tikzcd}
	& 0 \\
	1 & 2 & 3 \\
	& 4 \\
	& 5
	\arrow[from=1-2, to=2-1]
	\arrow[from=1-2, to=2-2]
	\arrow[from=1-2, to=2-3]
	\arrow[from=2-1, to=4-2]
	\arrow[from=2-2, to=3-2]
	\arrow[from=2-3, to=4-2]
	\arrow[from=3-2, to=4-2]
\end{tikzcd}\]
    A directed lattice $\Gamma$ over $L$  is a labeling of the vertices by unital commutative rings and a labeling of the arrows by unital ring homomorphisms. For instance, the following labelling
    
    
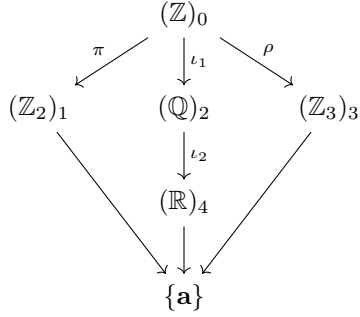
\[\begin{tikzcd}
	& (\z)_0 \\
	{(\z_2)_1} & (\q)_2 & {(\z_3)_3} \\
	& (\r)_4 \\
	& {\{\abf\}}
	\arrow["\pi"', from=1-2, to=2-1]
	\arrow["{\iota_1}", from=1-2, to=2-2]
	\arrow["\rho", from=1-2, to=2-3]
	\arrow[from=2-1, to=4-2]
	\arrow["{\iota_2}", from=2-2, to=3-2]
	\arrow[from=2-3, to=4-2]
	\arrow[from=3-2, to=4-2]
\end{tikzcd}\]
where $\pi$ and $\rho$ are the projection homomorphisms to the quotients $\z/(2)\simeq \z_2$ and $\z/(3)\simeq \z_3$ and $\iota_1$ and $\iota_2$ are the inclusion homomorphisms (note that $\iota_2\circ \iota_1$ is the inclusion map of $\z$ into $\r$). The unlabelled arrows correspond to the unique map that exists to the trivial ring. Since this map is unique we chose to omit its label.
\end{exmp}

We denote by $R^{\times}$ the set of invertible elements of the ring $R$. The following theorem from \cite{Dias_Dinis(23)} provides a link between rings and common meadows.

\begin{theorem}\label{T:DirectedLattice}
        Let $\Gamma$ a directed lattice of rings over a lattice  $L$ such that, for all $i\in I$ and all $x\in \Gamma_i$ the set:
        $$I_x=\{j\in I\mid f_{j,i}(x)\in \Gamma_j^{\times}\}$$
        has a greatest element.
          Then $M=\bigsqcup_{i\in L}\Gamma_i$ with the operations 
          \begin{itemize}
                \item $x+_My=f_{i\wedge j,i}(x)+_{i\wedge j}f_{i\wedge j,j}(y)$, where $+_{i\wedge j}$ is the sum in $\Gamma_{i\wedge j}$;
                \item $x \cdot_M y=f_{i\wedge j,i}(x)\cdot_{i\wedge j}f_{i\wedge j,j}(y)$, where $\cdot_{i\wedge j}$ is the product in $\Gamma_{i\wedge j}$.
            \end{itemize}
             is a common meadow such that $0\cdot M$ is lattice-equivalent to $L$.
\end{theorem}

\begin{exmp}
    Consider the following directed lattice of rings
    
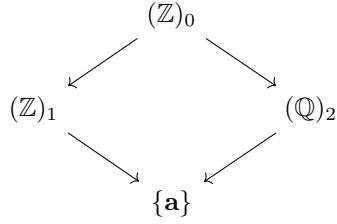
\[\begin{tikzcd}
	& {(\z)_0} \\
	{(\z)_1} && {(\q)_2} \\
	& {\{\abf\}}
	\arrow[from=1-2, to=2-1]
	\arrow[from=1-2, to=2-3]
	\arrow[from=2-1, to=3-2]
	\arrow[from=2-3, to=3-2]
\end{tikzcd}\]
    Let $x\in (\z)_0$ and $y\in (\q)_2$. Then by Theorem \ref{T:DirectedLattice} we have that $x+_M y= x+y \in (\q)_2$, since $2=1\wedge 2$. In the figure, the sum (and the product) of two elements can be calculated by r ``going down" on the lattice until the vertex where they first meet. In particular, the sum (and the product) of any element of $(\z)_1$ with any element of $(\q)_2$ will always give $\abf$.
\end{exmp}
        From now on, and to increase readability, we will drop the index in $+_M$ and $\cdot_M$.  We will also often drop the indices in the labels of the vertices.
        
    \begin{corollary}\label{C:DirectedLattice}
    Let $\Gamma$ be a directed lattice of rings over a lattice $L$. Then, there exists $M=\bigsqcup_{i\in I}\Gamma_i$, a pre-meadow with $\abf$, such that the lattice $0\cdot M$ is equivalent to $L$.
    \end{corollary}
    \begin{definition}
        Let $P$ be a pre-meadow with $\abf$. We define a partial order on $0\cdot P$ as follows: $z\leq z'$ if $z\cdot z'= z$, for all $z,z'\in 0\cdot P$. 
    \end{definition}
        \begin{proposition}\label{P:Transitionmaps}
            Let $M$ be a pre-meadow with $\abf$. If $z,z'\in 0\cdot M$ are such that $z\leq z'$,  then the map
            \begin{align*}
                f_{z,z'}:M_{z'}&\rightarrow M_z\\
                        x&\mapsto x+z
            \end{align*}
            is a ring homomorphism.

            Moreover, if $z,z',z''\in 0\cdot M$ are such that $z''\leq z'\leq z$, then $f_{z,z'}\circ f_{z',z''}=f_{z,z''}$.
        \end{proposition}
        The following proposition is useful to show immediately that certain directed lattices are common meadows, since it entails that if all the vertices of the associated lattice of a pre-meadow with $\abf$ are fields, then it defines a common meadow.

        \begin{corollary}\label{P:VerticesAreFields}
            Let $P$ be a pre-meadow with $\abf$ such that  $P_{0\cdot z}$ is a field, for all $0\cdot z\in 0\cdot P$. Then $P$ is a common meadow.
        \end{corollary}
        \begin{proof}
            Let $P$ be a pre-meadow with $\abf$ such that $P_{0\cdot z}$ is a field, for all $0\cdot z\in 0\cdot P$. Given an $x\in P\setminus 0\cdot P$, we have that if $x$ is invertible in $P_{0\cdot x}$, then the set $I_x$ has $0\cdot x$ as a maximal element. If $x\in 0\cdot P$ then $I_x=\{\abf\}$ which clearly has a maximal element. Hence $P$ is a common meadow by Theorem \ref{T:DirectedLattice}.
        \end{proof}

\section{From rings to Meadows}\label{S:MeadowstoRings}

   In this section we provide a canonical construction $M(R)$, of a pre-meadow with $\abf$, starting from a unital commutative ring $R$. If the ring is Artinian, then the pre-meadow with $\abf$ is shown to be a common meadow (see Theorem \ref{P:FunctorObject} below). We recall that a ring is said to be \emph{Artinian} if every descending sequence of ideals eventually stabilizes.

    Let $P$ be a pre-meadow with $\abf$ and $0\cdot z\in 0\cdot P$. Then there is a transition map $f_{0\cdot z,0}:P_0\to P_{0\cdot z}$ defined by $x\mapsto x+0\cdot z$, which by Proposition \ref{P:Transitionmaps} is a ring homomorphism. By the first isomorphism theorem  $P_0/\ker(f_{0\cdot z,0})$ is isomorphic to a subring of $P_{0\cdot z}$, namely the image of the map $f_{0\cdot z,0}$.

    A sort of inverse construction is also possible: starting with a unital commutative ring  $R$ and considering all the quotients of $R$ by its ideals. The following example illustrates this idea. 

    \begin{exmp}\label{E:Zideals}
       Recall that all the ideals $I$ of the integer ring $\z$ are of the form $I=(n)$, where $n\in \n$, i.e.\ the ideals are generated by a unique integer. The set $\{(n)\mid n\in\n\}$ is therefore ordered by inclusion (and divisibility), and so we can create a lattice of such ideals, partially represented below
        
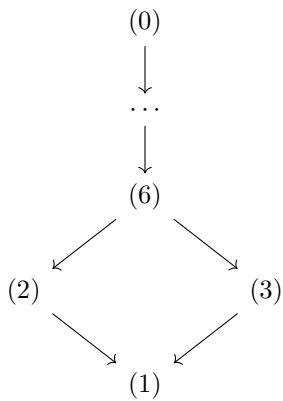
\[\begin{tikzcd}
	& (0) \\
	& \cdots \\
	& (6) \\
	(2) && (3) \\
	& (1)
	\arrow[from=1-2, to=2-2]
	\arrow[from=2-2, to=3-2]
	\arrow[from=3-2, to=4-1]
	\arrow[from=3-2, to=4-3]
	\arrow[from=4-1, to=5-2]
	\arrow[from=4-3, to=5-2]
\end{tikzcd}\]
    Note that $(2)+(3)=(1)= \z$, $(2)\cap (3)=(6)$ and $(0)=0$ (the trivial ring). Given this specific lattice we can construct the following directed lattice of rings
    
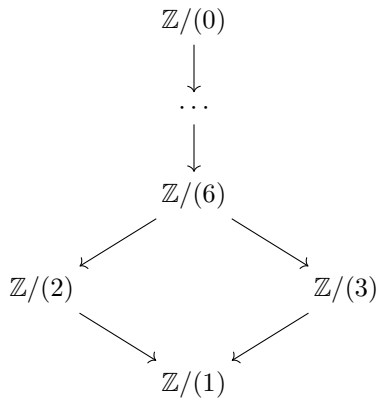
\[\begin{tikzcd}
	& {\z/(0)} \\
	& \cdots \\
	& {\z/(6)} \\
	{\z/(2)} && {\z/(3)} \\
	& {\z/(1)}
	\arrow[from=1-2, to=2-2]
	\arrow[from=2-2, to=3-2]
	\arrow[from=3-2, to=4-1]
	\arrow[from=3-2, to=4-3]
	\arrow[from=4-1, to=5-2]
	\arrow[from=4-3, to=5-2]
\end{tikzcd}\]
    where the transition maps are the natural projections, and $\z/(1)\simeq 0$ and $\z/(0)\simeq\z$. However, this directed lattice does not define a common meadow. Indeed, the class  $[2]_m\in\z/(m)\simeq \z_m$ is invertible if and only if $2$ does not divide $m$. Since the set $J_{[2]_0}=\{m\in \n\mid 2\nmid m\}$ has no greatest element, by Theorem~\ref{T:DirectedLattice} the directed lattice is not associated with a common meadow.
    \end{exmp}

    Example \ref{E:Zideals} also shows that the constructed pre-meadow with $\abf$ is not necessarily a common meadow. However, that is always the case if we restrict ourselves to Artinian rings.  In order to show that, we first prove the following lemma.

\begin{lemma}\label{L:InverseIntersection}
    Let $R$ be a unital commutative ring, and $I,J\trianglelefteq R$ be ideals of $R$. Let $x\in R$. Then the following are equivalent:
    \begin{enumerate}
        \item $x+I$ and $x+ J$ are invertible in the rings $R/I$ and $R/J$, respectively.
        \item $x+I\cap J$ is invertible in the ring $R/I\cap J$.
    \end{enumerate}
\end{lemma}
\begin{proof}
    Suppose first that $x+I$ and $x+ J$ are invertible in the rings $R/I$ and $R/J$, respectively. That is, there exist $y_1,y_2\in R$ such that $x\cdot y_1-1\in I$ and $x\cdot y_2-1\in J$. 

    Then 
    $$
       ( x\cdot y_1-1 )\cdot (x\cdot y_2-1)=x^2\cdot y_1\cdot y_2-x\cdot (y_1+y_2)+1=x\cdot (x\cdot y_1\cdot y_2 -(y_1+y_2) )+1 \in I\cap J.
    $$ 
    In particular, since $I\cap J$ is an ideal we have that
    $$-(x\cdot (x\cdot y_1\cdot y_2 -(y_1+y_2) )+1)=x\cdot (-x\cdot y_1\cdot y_2 +(y_1+y_2) )-1\in I\cap J .$$

    Hence $x+I\cap J$ is invertible, with inverse $(-x\cdot y_1\cdot y_2 +(y_1+y_2) )$.

    The other implication follows immediately from the existence of ring homomorphisms -- the projection maps -- from $R/I\cap J$ to $R/I$, and from $R/I\cap J$ to $R/J$.
\end{proof}

\begin{theorem}\label{P:FunctorObject}
    Let $R$ be a unital commutative ring. Then  $M(R)=\bigsqcup_{I\trianglelefteq R}R/I$, with the operations defined by
    \begin{itemize}
        \item $(x+I)+ (y+J)=(x+y)+I+J$
        \item $(x+I)\cdot (y+J)=(x\cdot y)+I+J$
        \item $-(x+I)=-x+I$,
    \end{itemize}
    is a pre-meadow with $\abf$. Furthermore, if $R$ is Artinian then $M(R)$ is a common meadow.
\end{theorem}
\begin{proof}
    It is straightforward to check that the operations are well defined and that they define a pre-meadow with $\abf$.
  
    Suppose now that $R$ is Artinian and let us see that $M(R)$ is a common meadow. By Theorem \ref{T:DirectedLattice} it is enough to prove that the set $I_x$ has a greatest element, for all $x\in M(R)$. For each $x\in R$ consider the set
$$
    I_x=\{I\trianglelefteq R\mid x+I\in (R/I)^\times\}.
$$

Since $R$ is Artinian we have that $I_x$ has at least one maximal element (note that the partial order in $I_x$ coincides with the reversed partial order given by the inclusion). Suppose now that $I_x$ has at least two maximal elements $I,J\trianglelefteq R$. By Lemma \ref{L:InverseIntersection} we have that $I\cap J\in I_x$, and $I\cap J \geq I, J$. Then by the maximality of $I$ and $J$ we must have $I=I\cap J$ and $J=I\cap J$, and so $I=J$. Hence there exists a unique maximal element. To see that for all $I\trianglelefteq R$ and $x\in R$ the set $I_{x+I}$ has a greatest element is similar. We then conclude that $M(R)$ is a common meadow, by Theorem~\ref{T:DirectedLattice}.
\end{proof}

Let us illustrate Theorem~\ref{P:FunctorObject} with some examples.

\begin{exmp}
\begin{enumerate}
    \item 
 Consider the finite ring $\z_6$ (which is clearly Artinian). By Theorem~\ref{P:FunctorObject} the following directed lattice, where the transition maps are the quotient maps, represents a common meadow
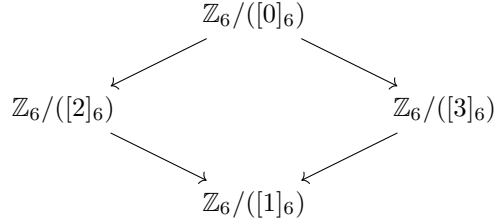
\[\begin{tikzcd}
	& {\z_6/([0]_6)} \\
	{\z_6/([2]_6)} && {\z_6/([3]_6)} \\
	& {\z_6/([1]_6)}
	\arrow[from=1-2, to=2-1]
	\arrow[from=1-2, to=2-3]
	\arrow[from=2-1, to=3-2]
	\arrow[from=2-3, to=3-2]
\end{tikzcd}\]
   \item It is well-known that if $\Bbbk$ is a field, then the polynomial ring $\Bbbk[x]$ is not Artinian. In fact, $1+x$ is invertible in $\Bbbk[x]/I$ if and only if $I$ is not the zero ideal. Then $I_{1+x}$ is equivalent to the set of all non zero ideals of $\Bbbk[x]$, and since it is not Artinian $I_{1+x}$ does not have a greatest element. In particular, $M(\Bbbk[x])$ is not a common meadow.
     However, given any non-trivial ideal $I$ of the ring $\Bbbk[x]$, the quotient ring $\Bbbk[x]/I$ is Artinian. Take for instance $I=(x^2)$. Then the direct lattice associated with this ring

    
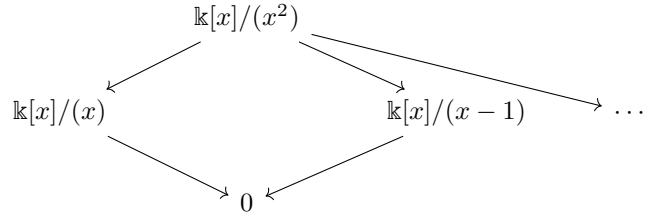
\[\begin{tikzcd}
	& {\Bbbk[x]/(x^2)} \\
	{\Bbbk[x]/(x)} && {\Bbbk[x]/(x-1)} & \cdots \\
	& 0
	\arrow[from=1-2, to=2-1]
	\arrow[from=1-2, to=2-3]
	\arrow[from=1-2, to=2-4]
	\arrow[from=2-1, to=3-2]
	\arrow[from=2-3, to=3-2]
\end{tikzcd}\]
    represents a common meadow by Theorem~\ref{P:FunctorObject}.
    \item One can do a similar construction with finite abelian groups. Let $R$ be a unital commutative ring and $A$ a finite abelian group. The set of formal sums of elements in $A$ with coefficients in $R$
    $$R[A]=\left\{ \sum_{g\in A}r_g g\mid r_g\in R\right\}$$  
    is called the \emph{group algebra} of $A$ (see \cite{curtis1966representation} for more on  group algebras). The group algebra is a useful object when studying representation theory of groups. We have that if $H\trianglelefteq A$ is a (normal) subgroup, then the map 
    \begin{align*}
        \varphi:R[A]&\to R[A/N]\\
                \sum_{g\in A}r_g g &\mapsto\sum_{g\in A} r_ggN
    \end{align*}
    is a ring homomorphism, such that $\ker(\varphi)$ is the ideal generated by $\{h-1\mid h\in H\}$. Since $\varphi$ is clearly surjective we have that $$R[A]/(h-1\mid h\in H)\simeq R[A/H].$$ We can then define the common meadow $M_R(A)=(\bigsqcup_{H\leq A}R[A/N])\bigsqcup \{\abf\}$. Take for instance the cyclic group $\z_{12}$. Its subgroups are generated by $1,2,3,4$ and $0$ (the non-negative divisors of $12$) then $M_{\cc}(\z_{12})$ is defined by the following directed lattice:
    
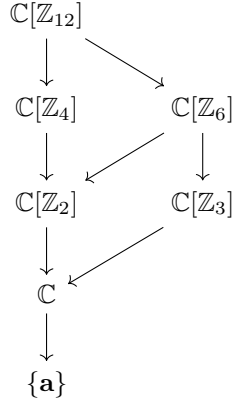
\[\begin{tikzcd}
	{\cc[\z_{12}]} \\
	{\cc[\z_4]} & {\cc[\z_6]} \\
	{\cc[\z_2]} & {\cc[\z_3]} \\
	\cc \\
	{\{\abf\}}
	\arrow[from=1-1, to=2-1]
	\arrow[from=1-1, to=2-2]
	\arrow[from=2-1, to=3-1]
	\arrow[from=2-2, to=3-1]
	\arrow[from=2-2, to=3-2]
	\arrow[from=3-1, to=4-1]
	\arrow[from=3-2, to=4-1]
	\arrow[from=4-1, to=5-1]
\end{tikzcd}\]
Note that if $m>0$ divides $12$ (i.e.\ $m=6,4,3,2,1$) we have that $m$ generates a subgroup $\langle m \rangle$ such that $\z_{12}/\langle m \rangle$ is isomorphic to $\z_{12/m}$. So the transition maps $\cc[\z_{12}]\to \cc[\z_{12/m}]$ are defined by $g\mapsto g\langle m \rangle$.

    \end{enumerate}
\end{exmp}

\begin{proposition}\label{P:MeadowsM0}
    Let $R$ and $S$ be commutative unital rings. Then $M(R)$ is isomorphic to $M(S)$ if and only if $M(R)_0=R$ is isomorphic to $M(S)_0=S$.
\end{proposition}

\begin{proof}
If $M(R)$ is isomorphic to $M(S)$, then in particular $M(R)_0=R$ and $M(S)_0=S$ are isomorphic rings.

Suppose now that $\psi:R\to S$ is a ring isomorphism. If $I$ is an ideal of $R$, then $\psi(I)$ is an ideal of $S$. Hence  
\begin{align*}
    \psi'&:M(R)\to M(S)\\
    &x+I \mapsto \psi(x)+\psi(I)    
\end{align*}
is an isomorphism of common meadows.
\end{proof}

\begin{exmp}
    Proposition \ref{P:MeadowsM0} implies the existence of  pre-meadows with $\abf$, say $P$, which are completely determined by $P_0$. However this is not always the case. Take for example the common meadows $M$ and $M'$ associated with the following lattices
    
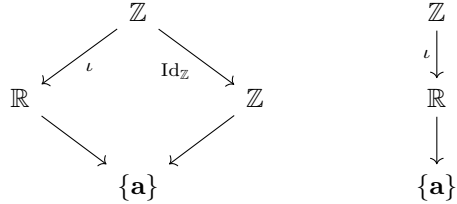
\[\begin{tikzcd}
	& \z &&& \z \\
	\r && \z && \r \\
	& {\{\abf\}} &&& {\{\abf\}}
	\arrow["\iota", from=1-2, to=2-1]
	\arrow["{\Id_{\z}}"', from=1-2, to=2-3]
	\arrow["\iota"', from=1-5, to=2-5]
	\arrow[from=2-1, to=3-2]
	\arrow[from=2-3, to=3-2]
	\arrow[from=2-5, to=3-5]
\end{tikzcd}\]
    Even though $M_0$ and $M'_0$ are isomorphic,  
(as rings) $M$ and $M'$ are not isomorphic since they have different lattice structures.
\end{exmp}

\begin{definition}
    We say that a common meadow $M$ is \emph{Artinian} if it is isomorphic to $M(R)$, for some Artinian ring $R$.
\end{definition}

We now show that every surjective ring homomorphism between Artinian rings gives rise to an homomorphism of common meadows (Proposition~\ref{P:FunctorMorphism}). This homomorphism of common meadows can then be used to define a functor between a subcategory of Artinian rings and the category of common meadows (Corollary~\ref{T:Functor_Art}).

\begin{proposition}\label{P:FunctorMorphism}
    Let $f:R\to S$ be a surjective ring homomorphism of unital commutative Artinian rings. Then the map $\overline{f}:M(R)\to M(S)$  defined by $x+I\mapsto f(x) + f(I)$ is a common meadow homomorphism. 
\end{proposition}

\begin{proof}
    The result is an immediate consequence of the fact that a surjective homomorphism preserves ideals.
\end{proof}

\begin{exmp}
    In Proposition \ref{P:FunctorMorphism}, the condition of $f:R\to S$ being surjective is necessary. In order to see this, consider for example the inclusion $i:\z\to \q$. Since $\q$ is a field, the only ideals are the trivial ones, there is no ring homomorphism from $\z/(m)\to \q$. 
\end{exmp}

Let $\mathrm{Ring^{Art}}$ be the category whose objects are unital commutative Artinian rings, and the morphisms are surjective ring homomorphisms and $\M_d$ be the category of common meadows (see \cite{Dias_Dinis(23)}). As a  consequence of Theorem \ref{P:FunctorObject} and Proposition~\ref{P:FunctorMorphism} we have the following result.

\begin{corollary}\label{T:Functor_Art}
There is a functor $M:\mathrm{Ring^{Art}}\to \M_d$ defined by $R\mapsto M(R)$.
\end{corollary}

\section{A decomposition theorem} \label{S:Decomposition}

The class of unital commutative Artinian rings is a well-behaved class of rings. For example, unital commutative Artinian rings admit a unique decomposition into unital commutative Artinian local rings (see e.g.\ \cite[Corollary 2.16]{eisenbud2013commutative}). For more on Artinian rings we refer to \cite{eisenbud2013commutative}. Motivated by this we introduce the notion of \emph{local common meadow} and show that (similarly with what happens in the case of Artinian rings)  there exists a class of common meadows, that we call \emph{Artinian meadows}, which admits a unique decomposition into local common meadows.

 A ring is \emph{local} if it has a unique maximal ideal. If $R$ is a local ring and $I$ and $J$ are ideals of $R$ such that $I+J=R$, then $I=R$ or $J=R$, since if $I$ and $J$ are proper ideals then they are contained in the unique maximal ideal of $R$. 
        
\begin{definition}
    Let $P$ be a pre-meadow with $\abf$. We say that $P$ is \emph{local} if $$x+y=\abf \to x=\abf \lor y=\abf,$$ for all $x,y \in P$.
\end{definition}

\begin{exmp}
\begin{enumerate}
    \item 
    Consider the following directed lattice
    
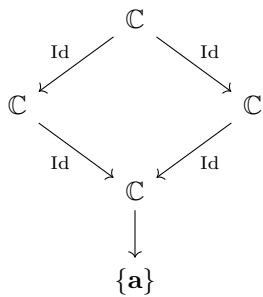
\[\begin{tikzcd}
	& \cc \\
	\cc && \cc \\
	& \cc \\
	& {\{\abf\}}
	\arrow["\Id"', from=1-2, to=2-1]
	\arrow["\Id", from=1-2, to=2-3]
	\arrow["\Id"', from=2-1, to=3-2]
	\arrow["\Id", from=2-3, to=3-2]
	\arrow[from=3-2, to=4-2]
\end{tikzcd}\]
Since the vertices are labelled by fields, the directed lattice defines a common meadow by Corollary \ref{P:VerticesAreFields}. It is in fact a local meadow because the product of any two elements different from $\abf$ will never be mapped to $\abf$.
\item  The common meadow associated with the following directed lattice
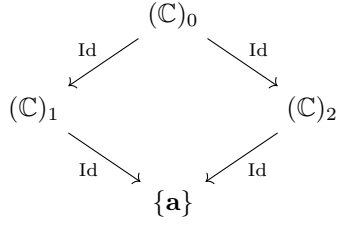
\[\begin{tikzcd}
	& (\cc)_0 \\
	(\cc)_1 && (\cc)_2 \\
	& {\{\abf\}}
	\arrow["\Id"', from=1-2, to=2-1]
	\arrow["\Id", from=1-2, to=2-3]
	\arrow["\Id"', from=2-1, to=3-2]
	\arrow["\Id", from=2-3, to=3-2]
\end{tikzcd}\]
is not local since we can take $x \in (\cc)_1$ and $y \in (\cc)_2$, which are clearly different from $\abf$, and their sum $x+y$ is $\abf$.
\end{enumerate}
\end{exmp}

    We now turn to show that if $R$ is Artinian, the common meadow $M(R)$ admits a unique decomposition into local meadows. 
    
    \begin{theorem}\label{T:Main}
    Let $R=\bigoplus R_i$ be the decomposition of a unital commutative Artinian ring $R$ into local rings $R_i$. Then $M(R)$ decomposes into the product of local meadows of the form $M(R_i)$ and this decomposition is unique up to permutation.
\end{theorem}

\begin{definition}
    Let $L$ be a lattice. We say that an element $a\in L$ is an \emph{atom} if $0\neq a$ and
    \begin{equation*}
 \forall x\in L \,(0\leq x \leq a   \to (x=0 \lor x=a)).
    \end{equation*}
    We say that $L$ is \emph{atomic} if for each $x\in L$ there is an atom $a$ such that $a\leq x$. 
\end{definition}

In a common meadow $P$, an element of $0\cdot P$ is an atom if it is not equal to $\abf$ and there is no element between the atom and $\abf$. We will say that $P$, a pre-meadow with $\abf$,  is \emph{atomic} if $0\cdot P$ is an atomic lattice. It is easy to see that all finite pre-meadows with $\abf$ are atomic.

\begin{exmp}\label{E:NotAtomicMeadow}
    Not all common meadows are atomic. Take for example $M=(\q\times \n)\sqcup \{\abf\}$, with addition and multiplication defined by
    \begin{itemize}
        \item $(\alpha,x)+(\beta,y)=(\alpha+\beta,\max\{x,y\})$
        \item $(\alpha,x)\cdot (\beta,y)=(\alpha\cdot \beta,\max\{x,y\})$
        \item $(\alpha,x)+\abf=\abf+(\alpha,x)=\abf$
        \item $(\alpha,x)\cdot \abf=\abf \cdot (\alpha,x)=\abf$.
    \end{itemize}
     A simple calculation shows that $M$ is a common meadow with $0\cdot M = (0\times \n)\sqcup \{\abf\}$, where the order is the usual order in $\n$. The lattice $0\cdot M$ has no atom since $\n$ has no maximal element, and so $M$ is not atomic.
\end{exmp}

Even though not all pre-meadows with $\abf$ are atomic, the pre-meadows with $\abf$ of the form $M(R)$ have atomic lattices, since every unital commutative ring $R$ has at least one maximal ideal. Since every ideal of $R$ is contained in a maximal ideal, we have that $M(R)$ is atomic.

The following lemma gives a characterization of local common meadows.
\begin{lemma}\label{L:LocalAtom}
    Let $M$ be an  atomic common meadow, then $M$ is local if and only if $0\cdot M$ has a unique atom. 
\end{lemma}
\begin{proof}
    Suppose that $M$ is local but there exist two distinct atoms $0\cdot z,0\cdot w \in 0\cdot M$. Since they are atoms we must have  $0\cdot w + 0\cdot z = \abf$. Since the common meadow is local either $0\cdot w =\abf$ or $0\cdot z = \abf$, which contradicts the hypothesis that they are atoms. 

    Suppose now that $0\cdot M$ has a unique atom, say $0\cdot z$. Let $x,y\in M$ be such that $x+y=\abf$ and let us first assume that $x\neq \abf$ and $y\neq \abf$. Then, we must have $0\cdot z \leq 0\cdot x$ and $0\cdot z \leq 0\cdot y$ with $0\cdot z$ the unique atom of $0\cdot M$. In particular,  $0\cdot z \leq 0\cdot (x+y)=\abf$, and so $0\cdot z$ cannot be a atom. Hence  $x=\abf$ or $y=\abf$, and so $M$ is local.
\end{proof}

The condition of being atomic is essential to have the equivalence given in Lemma~\ref{L:LocalAtom}. Indeed, the common meadow in Example \ref{E:NotAtomicMeadow} is a local meadow which is not atomic since  the associated  lattice has no atoms.

The term \emph{local meadow} is justified by the following connection with local rings. 

\begin{lemma}\label{L:LocalMeadow}
    If $R$ is a commutative Artinian ring, then $R$ is a local ring if and only if $M(R)$ is a local common meadow.
\end{lemma}
\begin{proof}
Given a unital commutative ring $R$, the atoms of the lattice $0\cdot M(R)$ are in bijection with the maximal ideals of $R$. By definition we have that $R$ is local if and only if it has a unique maximal ideal, which is equivalent to $0\cdot M(R)$ having a unique atom. By Lemma~\ref{L:LocalAtom} we then have that $R$ is local if and only if $M(R)$ is local, as we wanted.
\end{proof}

Note that if $R$ is a unital commutative local ring then $M(R)$ may fail to be a common meadow. However, it will still be a local pre-meadow with $\abf$.

\begin{exmp}
    Consider the pre-meadow with $\abf$ defined by the following directed lattice
    
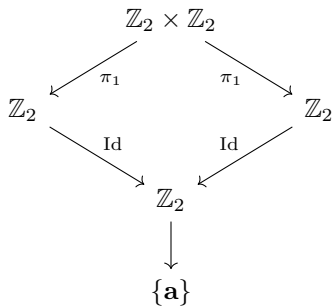
\[\begin{tikzcd}
	& {\z_2\times\z_2} \\
	{\z_2} && {\z_2} \\
	& {\z_2} \\
	& {\{\abf\}}
	\arrow["{\pi_1}", from=1-2, to=2-1]
	\arrow["{\pi_1}"', from=1-2, to=2-3]
	\arrow["\Id", from=2-1, to=3-2]
	\arrow["\Id"', from=2-3, to=3-2]
	\arrow[from=3-2, to=4-2]
\end{tikzcd}\]
    Since the set $I_{(1,0)}$ has two different maximal elements we have that this pre-meadow with $\abf$ is not a common meadow. Nevertheless it is a local pre-meadow since it has a unique atom.
\end{exmp}

Before we progress we recall the following property of ideals of unital commutative rings. 

\begin{proposition}[\cite{hungerford2012algebra}[Exercise III.3.22]\label{T:StructureIdeals}
    Let $R_1$ and $R_2$ be unital commutative rings. Then the ideals of $R_1\times R_2$ are of the form $I\times J$ with $I$ an ideal of $R_1$ and $J$ an ideal of $R_2$.
\end{proposition}

If $P$ and $Q$ are pre-meadows with $\abf$ we can define their Cartesian product  $P \times Q$ in a way that it is a pre-meadow with $\abf$. In \cite[Section 4]{Dias_Dinis(23)} it is claimed, without a proof, that the
Cartesian product of meadows is again a meadow. We show that the product $P \times Q$ of pre-meadows with $\abf$ is a common meadow if and only if $P$ and $Q$ are themselves common meadows. We start with a lemma which gives a relation between the structure of $P$ and $Q$ with the product $P \times Q$.

\begin{lemma}\label{L:PreMeadowProduct}
    Let $P$ and $Q$ be pre-meadows with $\abf$. Then $P\times Q$ is a pre-meadow with $\abf$ such that  $(P\times Q)_{(0\cdot z,0\cdot w)}=P_{0\cdot z}\times P_{0\cdot w}$. In particular we have
    $$P\times Q= \bigsqcup_{(0\cdot z,0\cdot w) \in (0\cdot P)\times (0\cdot Q)}P_{0\cdot z}\times Q_{0\cdot w}.$$ 
\end{lemma}
\begin{proof}
    We start by showing that $(P\times Q)_{(0\cdot z,0\cdot w)}=P_{0\cdot z}\times P_{0\cdot w}$. Let $(x,y)\in (P\times Q)_{(0\cdot z,0\cdot w)}$. Then
    $$
    (0\cdot z,0\cdot w)=(x,y)\cdot (0,0)=(0\cdot x,0\cdot y),
    $$
    and so $(x,y)\in P_{0\cdot z}\times P_{0\cdot w}$. Similarly, if we take a $(x,y)\in P_{0\cdot z}\times P_{0\cdot w}$, then 
    $$
    (x,y)\cdot (0,0)=(0\cdot x,0\cdot y) =(0\cdot z, 0\cdot w), 
    $$
    and so  $(P\times Q)_{(0\cdot z,0\cdot w)}=P_{0\cdot z}\times P_{0\cdot w}$. 

    The result then follows from the fact that
    \begin{equation*}
        \begin{split}
                P\times Q &= \bigsqcup_{(0\cdot z,0\cdot w) \in (0\cdot P)\times (0\cdot Q)} (P\times Q)_{(0\cdot z,0\cdot w)}\\
                &=  \bigsqcup_{(0\cdot z,0\cdot w) \in (0\cdot P)\times (0\cdot Q)} P_{0\cdot z}\times P_{0\cdot w}. \qedhere
        \end{split}
    \end{equation*}
\end{proof}

Using Lemma \ref{L:PreMeadowProduct} we can prove that any finite product of common meadows is still a common meadow.

\begin{theorem} \label{P:CommonMeadowProduct}
    Let $P$ and $Q$ be pre-meadows with $\abf$ then $P\times Q$ is a common meadow if and only if $P$ and $Q$ are common meadows.
\end{theorem}

\begin{proof}
    Suppose that $P$ and $Q$ are common meadows.    By Theorem \ref{T:DirectedLattice} we have that $P\times Q$ is a common meadow if for all $(x,y)\in P\times Q$ the  set
    $$I_{(x,y)}=\{(0\cdot z,0\cdot w)\in 0\cdot P\times 0\cdot Q\mid (x+0\cdot z,y+0\cdot w)\in (P_{0\cdot z}\times Q_{0\cdot w})^\times\}$$

    has a greatest element. In fact, by Lemma \ref{L:PreMeadowProduct} we have that $(P\times Q)_{(0\cdot z,0\cdot w)}=P_{0\cdot z}\times P_{0\cdot w}$, and so an element $(x+0\cdot z,y+0\cdot w)$ is invertible in $P_{0\cdot z}\times P_{0\cdot w}$ if and only if $x+0\cdot z$ is invertible  in $P_{0\cdot z}$ and $y+0\cdot w$ is invertible  in $Q_{0\cdot w}$, that is we have that $I_{(x,y)}=I_x\times I_y$. 
    
    Since $P$ and $Q$ are common meadows, $I_x$ and $I_y$ both have a greatest element 
 and so $I_{(x,y)}$ must also have a greatest element. We conclude that $P\times Q$ is a common meadow.

    Suppose now that $P\times Q$ is a common meadow,  and let us see that $P$ is a common meadow. Consider the function $\psi:P\to P\times Q$ defined by $\psi(x)=(x,\abf)$. It is straightforward to check that $\psi$ is an homomorphism of semigroups for both operations. In particular, the restriction $\tilde{\psi}:P_{0\cdot z}\to P_{0\cdot z}\times\{\abf\}$ is a ring homomorphism that commutes with the transition maps. Then, given $x\in P$ and $0\cdot z\in 0\cdot P$ (such that $0\cdot z\cdot x=0\cdot z$) we have that $x+0\cdot z\in P_{0\cdot z}^\times$ if and only if $P_{0\cdot z}^\times\times \{\abf
    \}$. Then $I_{(x,\abf)}=I_x\times\{\abf\}$ and since $P\times Q$ is a common meadow we have that $I_{(x,\abf)}=I_x\times\{\abf\}$ has a greatest element and so $I_x$ also has a greatest element and consequently $P$ is a common meadow. The argument for $Q$ is similar. 
\end{proof}

Next we will see that a decomposition of a unital commutative ring $R$ gives rise to a decomposition of $M(R)$. In particular, common meadows of the form $M(R)$ with $R$ an unital commutative Artinian ring decompose into a product of local common meadows.

\begin{theorem}\label{L:MeadowDecomposition}
    Let $R=\bigoplus_{i=1}^{n}R_i$ be a decomposition of a unital commutative ring. Then $$M(R)\simeq \prod_{i=1}^nM(R_i).$$
\end{theorem}
\begin{proof}
We start by showing that 
\begin{equation}\label{E:isomorfic}
M(R)\simeq \bigsqcup_{I_i\trianglelefteq R_i}\prod_{i=1}^n R_i/I_i .
\end{equation}

We consider only the case $R=R_1\times R_2$, since the general case then follows by an easy induction on $n$. By Proposition \ref{T:StructureIdeals}, the ideals of $R$ are of the form $I\times J$, where $I\trianglelefteq R_1$ and $J\trianglelefteq R_2$. 

We have that
\begin{equation*}
\begin{split}
M(R)&=\bigsqcup_{K\trianglelefteq R_1\times R_2} (R_1\times R_2)/K\\
&=\bigsqcup_{I\trianglelefteq R_1,J\trianglelefteq R_2} (R_1\times R_2)/(I\times J)\\
&\simeq \bigsqcup_{I\trianglelefteq R_1,J\trianglelefteq R_2} (R_1/I\times R_2/J).    
\end{split}
\end{equation*}

Where the isomorphism is given by  $(x,y)+I\times J\mapsto (x+I,y+J)$. Hence $M(R)\simeq \bigsqcup_{I_i\trianglelefteq R_i}\prod_{i=1}^n R_i/I_i $.

Now, by Lemma \ref{L:PreMeadowProduct} and equation \eqref{E:isomorfic} we have that
$$\prod_{i=1}^nM(R_i) = \bigsqcup_{I_i\trianglelefteq R_i}\prod_{i=1}^n R_i/I_i\simeq M(R),$$
which entails the result.
\end{proof}

By Theorem \ref{P:CommonMeadowProduct} we have that $M(R)$ is a common meadow if and only if $M(R_i)$ is a common meadow for all $i\in\{1,\cdots,n\}$.

\begin{definition}
    Let $M$ be a common meadow. If $M\simeq \prod_{i=1}^n M_i$ with $M_i$ local meadows, we say that $M$ \emph{decomposes} into (the product of) local meadows. 
\end{definition}

    We are now able to prove the first part of the main result.

\begin{corollary}\label{Cor:Decomposition}
    If $R$ is a commutative unital Artinian ring, then $M(R)$ decomposes into a finite product of local common meadows. 
\end{corollary}
\begin{proof}
Let $R$ be a unital commutative Artinian ring. Then, there exist local rings $R_i$ such that $R\simeq \bigoplus_{i=1}^{n}R_i$. By Lemma \ref{L:MeadowDecomposition} we have that $M(R)\simeq \prod_{i=1}^nM(R_i)$. Hence, each $M(R_i)$ is a local meadow, by Lemma  \ref{L:LocalMeadow}.
\end{proof}

Corollary~\ref{Cor:Decomposition} shows that in case $R$ is a unital commutative Artinian ring, the common meadow $M(R)$ admits a decomposition into local meadows. We now turn to show that such decomposition is essentially unique. 

The following result provides a useful characterization of atomic meadows in terms of its components.

\begin{lemma}\label{L:ProductofAtomicIsAtomic}
    Let $M$ be a common meadow, with $M\simeq \prod_{i=1}^t M_i$. Then $M$ is atomic if and only if $M_i$ is atomic for all $i \in \{1, \dots, t\}$.
\end{lemma}
\begin{proof}
    Let us consider the case $t=2$, i.e.\ $M\simeq M_1\times M_2$, and then $0\cdot M\simeq 0\cdot M_1\times 0\cdot M_2$. The general case follows by an easy induction on $t$.

    Suppose first that $M$ is atomic. Then, for each $(x,\abf)\in 0\cdot M_1$, there is an atom $(a,b)\in 0\cdot M_1\times 0\cdot M_2$ such that $(a,b)\leq (x,\abf)$. This means that $b=\abf$, since $\abf$ is the minimum, and $a\in 0\cdot M_1$ is an atom such that $a\leq x$, which means that $0\cdot M_1$ is atomic. With a similar argument one sees that $0\cdot M_2$ is atomic.

    Suppose now that $0\cdot M_1$ and $0\cdot M_2$ are atomic and take a $(x,y)\in  0\cdot M_1\times 0\cdot M_2$. Then there exist atoms $x_1\in 0\cdot M_1$ and $y_1\in 0\cdot M_1$ such that $x_1\leq x$ and $y_1\leq y$, which means that $(x_1,y_1)\leq (x,y)$. Since $x_1$ and $y_1$ are atoms we have that $(x_1,y_1)$ is also an atom and so $0\cdot M$ is atomic.    
\end{proof}

\begin{lemma}
    Every decomposition of an atomic common meadow into local meadows has the same number of factors.
\end{lemma}
\begin{proof}
Let $M$ be a common meadow with atomic lattice such that $M=\prod_{i=1}^t M_i$, where $M_i$ is a local meadow, for all $i\in \{1,\cdots,t\}$. In particular each $M_i$ is atomic by Lemma \ref{L:ProductofAtomicIsAtomic}, and has a unique atom, by Lemma \ref{L:LocalAtom}.

Let $a_i$ be the unique atom of $0\cdot M_i$ and $\abf_j$ the inverses of zero in each $M_j$. One can easily see that the atoms of $0\cdot M$ are of the form $\abf_1\times \cdots \times \abf_{i-1}\times a_i \times \abf_{i+1}\cdots \times \abf_t$, that is, the elements of $0\cdot M$ that are products of the $\abf_j$'s of the rings $M_j$, with the exception of exactly one occurrence of $a_i$. Then the number of atoms of $M$ is exactly $t$, and therefore any other decomposition into local meadows must also have $t$ factors.
\end{proof}
\begin{proposition}\label{P:UniqueDecomposition}
Let $R$ be a unital commutative Artinian ring that decomposes into local rings $R=\bigoplus_{i=1}^{t}R_i$. If $M(R)=\prod_{i=1}^t M_i$, then there is a permutation function $\pi\in S_t$ such that $M_i=M(R_{\pi(i)})$.
\end{proposition}

\begin{proof}
We again consider the case where $t=2$ since the general case follows by an easy induction on $t$. So, we may consider $R=S\times T$ and $M(R)\simeq M\times N$. The general case follows by an easy induction on $t$. We have that $M_0\times N_0\simeq S\times T$ is a decomposition into local rings, because if $M_0$ is not a local ring, then it would decompose into the product of local rings (since $M_0\simeq M_0\times N_0/0\times N_0$ is Artinian) and the decomposition into local rings would not be unique. Without loss of generality we can assume that $M_0\simeq S$ and $N_0\simeq T$. 

By hypothesis, $M\times N\simeq M(R)\simeq  M(S)\times M(T)$  so there is a bijection 
$$\psi:\{I\times J\mid I\trianglelefteq S, J\trianglelefteq T\}\to 0\cdot M\times 0\cdot N.$$ Hence the following diagram commutes
\[\begin{tikzcd}
	{S\times T} & {M_0\times N_0} \\
	{S/I\times T/J} & {M_{\psi(I\times J)_1}\times N_{\psi(I\times J)_2}}
	\arrow["\Phi_0", from=1-1, to=1-2]
	\arrow["\pi_I\times \pi_J",from=1-1, to=2-1]
	\arrow["f_{\psi(I\times J)_1}\times f_{\psi(I\times J)_2}",from=1-2, to=2-2]
	\arrow["\Phi_{I\times J}", from=2-1, to=2-2]
\end{tikzcd}\]
Here the horizontal arrows represent the isomorphism $\Phi:M(S)\times M(T)\to M\times N$. Note that $\Phi_0(S)=M_0$ and $\Phi_0(T)=N_0$.  Since the diagram is commutative we have that $\Phi_{I\times J}(S/I\times 0)=M_{\psi(I\times J)_1}$ and $\Phi_{I\times J}(0\times T/J)=N_{\psi(I\times J)_2}$ and hence 
 $M(S)\simeq M$ and $M(T)\simeq N$.
\end{proof}

The proof of Theorem \ref{T:Main} is now straightforward.
\begin{proof}[Proof of Theorem~\ref{T:Main}]
   Let $R$ be a unital commutative Artinian ring. By Corollary \ref{Cor:Decomposition} we have that $M(R)\simeq\prod_{i=1}^tM(R_i)$ and, by Proposition \ref{P:UniqueDecomposition}, this decomposition is unique up to permutation.
\end{proof}

    We would like to point out that not all common meadows decompose into a product of local meadows, as illustrated by the following example.

    \begin{exmp}
     Let $M$ be the common meadow defined by the following directed lattice
            
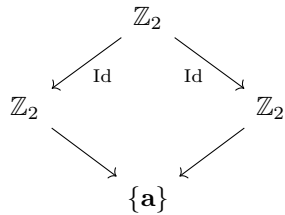
\[\begin{tikzcd}
	& {\z_2} \\
	{\z_2} && {\z_2} \\
	& {\{\abf\}}
	\arrow["\Id", from=1-2, to=2-1]
	\arrow["\Id"', from=1-2, to=2-3]
	\arrow[from=2-1, to=3-2]
	\arrow[from=2-3, to=3-2]
\end{tikzcd}\]
            If $M$ could be decomposed into $M\simeq N\times N'$, then $\z_2=N_0\times N_0'$ and so $N_0$ or $N'_0$ would need to be $0$ since $\z_2$ is a local ring.
             However, even if $M_0$ is not a local ring, it can happen that $M$ is indecomposable. Since a decomposition of a meadow also gives a lattice decomposition we have that $|0\cdot M|$ is prime then $M$ is indecomposable.
    \end{exmp}

    In conclusion, a decomposition of a meadow $M$ also gives a decomposition of the lattice $0\cdot M$ and a decomposition of the rings $M_{0\cdot x}$, for all $0\cdot x\in 0\cdot M$.

    We conclude the paper with some open questions:
Our first question is a sort of converse of Theorem~\ref{P:FunctorObject}.
\begin{question}
    Are there other classes of unital commutative rings for which $M(R)$ is a common meadow?  
\end{question}

\begin{question}
       Is it possible to characterize which common meadows decompose into the product of local meadows?
\end{question}

\bibliographystyle{siam}
\bibliography{References}

\begin{thebibliography}{10}

\bibitem{bergstra2020arithmetical}
{\sc J.~Bergstra}, {\em Arithmetical datatypes, fracterms, and the fraction
  definition problem}, Transmathematica,  (2020).

\bibitem{Bergstra2008}
{\sc J.~Bergstra, Y.~Hirshfeld, and J.~Tucker}, {\em Fields, Meadows and
  Abstract Data Types}, Springer Berlin Heidelberg, Berlin, Heidelberg, 2008,
  pp.~166--178.

\bibitem{Bergstra2015}
{\sc J.~Bergstra and A.~Ponse}, {\em Division by Zero in Common Meadows},
  Springer International Publishing, Cham, 2015, pp.~46--61.

\bibitem{BP(20)}
\leavevmode\vrule height 2pt depth -1.6pt width 23pt, {\em Arithmetical
  datatypes with true fractions}, Acta Inform., 57 (2020), pp.~385--402.

\bibitem{10.1145/1219092.1219095}
{\sc J.~Bergstra and J.~Tucker}, {\em The rational numbers as an abstract data
  type}, J. ACM, 54 (2007), p.~7–es.

\bibitem{bergstra2023axioms}
\leavevmode\vrule height 2pt depth -1.6pt width 23pt, {\em On the axioms of
  common meadows: Fracterm calculus, flattening and incompleteness}, The
  Computer Journal, 66 (2023), pp.~1565--1572.

\bibitem{Bergstra_Tucker_2024}
{\sc J.~A. Bergstra and J.~V. Tucker}, {\em Fracterm calculus for signed common
  meadows}, Transmathematica,  (2024).

\bibitem{bergstra2024ringscommondivisioncommon}
{\sc J.~A. Bergstra and J.~V. Tucker}, {\em Rings with common division, common
  meadows and their conditional equational theories}, 2024.

\bibitem{Dinis_Bottazzi}
{\sc E.~Bottazzi and B.~Dinis}, {\em Flexible involutive meadows}, 2023.
\newblock (to appear in the Journal of Applied Logics) arXiv: 2309.01284.

\bibitem{curtis1966representation}
{\sc C.~W. Curtis and I.~Reiner}, {\em Methods of representation theory. {V}ol.
  {I}}, Pure and Applied Mathematics, John Wiley \& Sons, Inc., New York, 1981.
\newblock With applications to finite groups and orders, A Wiley-Interscience
  Publication.

\bibitem{Dias_Dinis(23)}
{\sc J.~Dias and B.~Dinis}, {\em Strolling through common meadows},
  Communications in Algebra,  (2024), pp.~1--28.

\bibitem{Dias_Dinis(24)}
{\sc J.~Dias and B.~Dinis}, {\em Towards an enumeration of finite common
  meadows}, International Journal of Algebra and Computation,  (2024).

\bibitem{eisenbud2013commutative}
{\sc D.~Eisenbud}, {\em Commutative algebra: with a view toward algebraic
  geometry}, vol.~150, Springer Science \& Business Media, 2013.

\bibitem{hungerford2012algebra}
{\sc T.~W. Hungerford}, {\em Algebra}, vol.~73, Springer Science \& Business
  Media, 2012.

\end{thebibliography}
\subsection*{Acknowledgments}
Both authors acknowledge the support of FCT - Funda\c{c}\~ao para a Ci\^{e}ncia e Tecnologia under the project: 10.54499/UIDB/04674/2020, and the research center CIMA -- Centro de Investigação em Matemática e Aplicações. 

The second author also acknowledges the support of CMAFcIO -- Centro de Matem\'{a}tica, Aplica\c{c}\~{o}es Fundamentais e Investiga\c{c}\~{a}o Operacional under the project UIDP/04561/2020.

\end{document}